\documentclass[12pt,leqno]{article}

\usepackage{amssymb,amsfonts,amsmath,amsthm,amscd,mathrsfs}
\usepackage{graphics,graphicx}
\def\IE{{\mathbb E}}

\def\IP{{\mathbb P}}
\def\IR{{\mathbb R}}

\def\IZ{{\mathbb Z}}

\def\n{\noindent}
\def\dsl{\textstyle\sum\limits}

\def\dis{\displaystyle}
\def\o{\omega}
\def\fr{\mbox{\footnotesize $\dis\frac{1}{2}$}}

\def\f{\footnotesize}
\def\r{\rightarrow}
\def\point{{\mbox{\large $.$}}}
\def\wh{\widehat}
\def\wt{\widetilde}

\def\cA{{\cal A}}

\def\cT{{\cal T}}

\def\cI{{\cal I}}

\dimendef\dimen=0

\newtheorem{theorem}{Theorem}[section]

\newtheorem{remark}[theorem]{Remark}

\begin{document}

\baselineskip14pt
\noindent

\begin{center}
{\bf AN ISOMORPHISM THEOREM FOR RANDOM INTERLACEMENTS}
\end{center}

\vspace{0.5cm}

\begin{center}
Alain-Sol Sznitman$^*$
\end{center}

\bigskip
\begin{center}
Preliminary Draft
\end{center}

\bigskip
\begin{abstract}
We consider continuous-time random interlacements on a transient weighted graph. We prove an identity in law relating the field of occupation times of random interlacements at level $u$ to the Gaussian free field on the weighted graph. This identity is closely linked to the generalized second Ray-Knight theorem of \cite{EisKasMarRosShi00},~\cite{MarcRose06}, and uniquely determines the law of occupation times of random interlacements at level $u$.
\end{abstract}

\vspace{9cm}

\noindent
Departement Mathematik  \hfill  November 2011
\\
ETH-Zentrum\\
CH-8092 Z\"urich\\
Switzerland

\vfill 
\n
$\overline{~~~~~~~~~~~~~~~~~~~~~~~~~~~~~~~~~~}$

\n
{\footnotesize{$^*$ This research was supported in part by the grant ERC-2009-AdG  245728-RWPERCRI}}

\newpage

\thispagestyle{empty}
~

\newpage
\setcounter{page}{1}

\setcounter{section}{-1}
\section{Introduction}

In this note we consider continuous-time random interlacements on a transient weighted graph $E$. We prove an identity in law, which relates the field of occupation times of random interlacements at level $u$ to the Gaussian free field on $E$. The identity can be viewed as a kind of generalized second Ray-Knight theorem, see \cite{EisKasMarRosShi00}, \cite{MarcRose06}, and characterizes the law of the field of occupation times of random interlacements at level $u$.

\medskip
We now describe our results and refer to Section 1 for details. We consider a countable, locally finite, connected graph, with vertex set $E$, endowed with non-negative symmetric weights $c_{x,y} = c_{y,x}$, $x,y \in E$, which are positive exactly when $x,y$ are distinct and $\{x,y\}$ is an edge of the graph. We assume that the induced discrete-time random walk on $E$ is transient. Its transition probability is defined by 
\begin{equation}\label{0.1}
p_{x,y} = \dis\frac{c_{x,y}}{\lambda_x}, \;\;  \mbox{where $\lambda_x = \dsl_{z \in E} c_{x,y}$, for $x,y \in E$}.
\end{equation}

\n
In essence, continuous-time random interlacements consist of a Poisson point process on a certain space of doubly infinite $E$-valued trajectories marked by their duration at each step, modulo time-shift. A non-negative parameter $u$ plays the role of a multiplicative factor of the intensity of this Poisson point process, which is defined on a suitable canonical space $(\Omega, \cA, \IP)$. The field of occupation times of random interlacements at level $u$ is then defined for $x \in E$, $u \ge 0$, $\o \in \Omega$, by (see (\ref{1.8}) for the precise expression)
\begin{equation}\label{0.2}
\begin{split}
L_{x,u}(\o) = \lambda_x^{-1} \times &\; \mbox{the total duration spent at $x$ by the trajectories modulo}\\
&\;\mbox{time-shift with label at most $u$ in the cloud $\omega$.}
\end{split}
\end{equation}

\n
The Gaussian free field on $E$ is the other ingredient of our isomorphism theorem. Its canonical law $P^G$ on $\IR^E$ is such that
\begin{equation}\label{0.3}
\begin{array}{l}
\mbox{Under $P^G$, the canonical field $\varphi_x$, $x \in E$, is a centered Gaussian field with}\\
\mbox{covariance $E^{P^G} [\varphi_x \varphi_y] = g(x,y)$, for $x,y \in E$,}
\end{array}
\end{equation}

\n
where $g(\cdot,\cdot)$ stands for the Green function attached to the walk on $E$, see (\ref{1.3}). The main result of this note is the next theorem:

\begin{theorem}\label{theo0.1}
For each $u \ge 0$,
\begin{equation}\label{0.4}
\begin{array}{l}
\big(L_{x,u} + \fr  \;\varphi^2_x\big)_{x \in E} \;\mbox{under $\IP \otimes P^G$, has the same law as}
\\[1ex]
\big(\fr \;(\varphi_x + \sqrt{2u})^2\big)_{x \in E} \;\mbox{under $P^G$}.
\end{array}
\end{equation}
\end{theorem}

This theorem provides for each $u$ an identity in law very much in the spirit of the so-called generalized second Ray-Knight theorems, see Theorem 1.1 of \cite{EisKasMarRosShi00} or Theorem 8.2.2 of \cite{MarcRose06}. Remarkably, although we are in a transient set-up, (\ref{0.4}) corresponds to the recurrent case in the context of generalized Ray-Knight theorems. Let us underline that (\ref{0.4}) uniquely determines the law of $(L_{x,u})_{x \in E}$ under $\IP$, as the consideration of Laplace transforms readily shows. We also refer to Remark \ref{rem3.1} for a variation of (\ref{0.4}).

\medskip
The proof of Theorem \ref{theo0.1} involves an approximation argument of the law of $(L_{x,u})_{x \in E}$ stated in Theorem \ref{theo2.1}, which is of independent interest. This approximation has a similar flavor to what appears at the end of Section 4.5 of \cite{Szni11d}, when giving a precise interpretation of random interlacements as ``loops going through infinity'', see also \cite{Leja11}, p.~85. The combination of Theorem \ref{theo2.1} and the generalized second Ray-Knight theorem readily yields Theorem \ref{theo0.1}.  As an application of Theorem \ref{theo0.1} we give a new proof of Theorem 5.1 of \cite{Szni11c} concerning the large $u$ behavior of $(L_{x,u})_{x \in E}$, see Theorem \ref{theo4.1}.

\medskip
We now explain how this note is organized.

\medskip
In Section 1, we provide precise definitions and recall useful facts. Section 2 develops the approximation procedure for $(L_{x,u})_{x \in E}$. We give two proofs of the main Theorem \ref{theo2.1}, and an extension appears in Remark \ref{rem2.2}. The short Section 3 contains the proof of Theorem \ref{theo0.1}, and a variation of (\ref{0.4}) in Remark \ref{rem3.1}. In Section 4, we present an application to the study of the large $u$ behavior of $(L_{x,u})_{x \in E}$, see Theorem \ref{theo4.1}.

\section{Notation and useful results}
\setcounter{equation}{0}

In this section we provide additional notation and recall some definitions and useful facts related to random walks, potential theory, and continuous-time interlacements.

\medskip
We consider the spaces $\wh{W}_+$ and $\wh{W}$ of infinite, and doubly infinite, $E \times (0,\infty)$-valued sequences, such that the $E$-valued sequences form an infinite, respectively doubly-infinite, nearest-neighbor trajectory spending finite time in any finite subset of $E$, and such that the $(0,\infty)$-valued components have an infinite sum in the case of $\wh{W}_+$, and infinite ``forward'' and ``backward'' sums, when restricted to positive and negative indices, in the case of $\wh{W}$.

\medskip
We write $Z_n, \sigma_n$, with $n \ge 0$, or $n \in \IZ$, for the respective $E$- and $(0,\infty)$-valued coordinates on $\wh{W}_+$ and $\wh{W}$. We denote by $P_x$, $x \in E$, the law on $\wh{W}_+$, endowed with its canonical $\sigma$-algebra, under which $Z_n, n \ge 0$, is distributed as simple random walk starting at $x$, and $\sigma_n$, $n \ge 0$, are i.i.d. exponential variables with parameter $1$, independent from the $Z_n$, $n \ge 0$. We denote by $E_x$ the corresponding expectation. Further, when $\rho$ is a measure on $E$, we write $P_\rho$ for the measure $\sum_{x \in E} \rho(x) P_x$, and $E_\rho$ for the corresponding expectation. 

\medskip
We denote by $X_t$, $t \ge 0$, the continuous-time random walk on $E$, with constant jump rate $1$, defined for $t \ge 0$, $\wh{w} \in \wh{W}_+$, by
\begin{equation}\label{1.1}
X_t(\wh{w}) = Z_k(\wh{w}), \; \mbox{when} \; \sigma_0(\wh{w}) + \dots + \sigma_{k-1}(\wh{w}) \le t < \sigma_0 (\wh{w}) + \dots + \sigma_k (\wh{w})
\end{equation}

\medskip\n
(by convention the term bounding $t$ from below vanishes when $k = 0$).

\medskip
Given $U \subseteq E$, we write $H_U = \inf\{t \ge 0; X_t \in U\}$, $\wt{H}_U = \inf\{t > 0; X_t \in U$, and for some $s \in (0,t)$, $X_s \not= X_0\}$, and $T_U = \inf\{ t \ge 0; X_t \notin U\}$, for the entrance time in $U$, the hitting time of $U$, and the exit time from $U$. We denote by $g_U(\cdot,\cdot)$ the Green function of the walk killed when exiting $U$
\begin{equation}\label{1.2}
g_U(x,y) = \mbox{\f $\dis\frac{1}{\lambda_y}$} \;E_x \Big[\dis\int^{T_U}_0 1\{X_s = y\} ds\Big], \;\mbox{for $x,y \in E$}.
\end{equation}

\medskip\n
The function $g_U(\cdot,\cdot)$ is known to be symmetric and finite (due to the transience assumption we have made). When $U = E$, no killing takes place (i.e.~$T_U = \infty$), and we simply write
\begin{equation}\label{1.3}
g(x,y) = g_{U = E} (x,y), \; \mbox{for $x,y \in E$},
\end{equation}
for the Green function.

\medskip
Given a finite subset $K$ of $U$, the equilibrium measure and capacity of $K$ relative to $U$ are defined by
\begin{align}
& e_{K,U} (x) = P_x [\wt{H}_K > T_U] \,\lambda_x \,1_K(x), \; \mbox{for $x \in E$}, \label{1.4}
\\[1ex]
&{\rm cap}_U(K) = \dsl_{x \in E} e_{K,U} (x). \label{1.5}
\end{align}

\n
When $U = E$, we simply drop $U$ from the notation, and refer to $e_K$ and ${\rm cap}(K)$, as the equilibrium measure and the capacity of $K$. Further, the probability to enter $K$ before exiting $U$ can be expressed as 
\begin{equation}\label{1.6}
P_x[H_K <T_U] = \dsl_{x \in E} g_U (x,y) \,e_{K,U}(y), \;\mbox{for $x \in E$}.
\end{equation}

\n
We now turn to the description of continuous-time random interlacements on the transient weighted graph $E$. We write $\wh{W}^*$ for the space $\wh{W}$ (introduced at the beginning of this section), modulo time-shift, i.e.~$\wh{W}^* = W/\sim$, where for $\wh{w}$, $\wh{w}' \in \wh{W}$, $\wh{w} \sim \wh{w}'$ means that $\wh{w}(\cdot) = \wh{w}' ( \cdot + k)$ for some $k \in \IZ$. We denote by $\pi^*$: $\wh{W} \r \wh{W}^*$ the canonical map, and endow $\wh{W}^*$ with the $\sigma$-algebra consisting of sets with inverse image under $\pi^*$ belonging to the canonical $\sigma$-algebra of $\wh{W}$.

\medskip
The continuous-time interlacement point process is a Poisson point process on the space $\wh{W}^* \times \IR_+$. Its intensity measure has the form $\nu(d\wh{w}^*)du$, where $\wh{\nu}$ is the $\sigma$-finite measure on $\wh{W}^*$ such that for any finite subset $K$ of $E$, the restriction of $\wh{\nu}$ to the subset of $\wh{W}^*$ consisting of those $\wh{w}^*$ for which the $E$-valued trajectory modulo time-shift enters $K$, is equal to $\pi^* \circ \wh{Q}_K$, the image of $\wh{Q}_K$ under $\pi^*$, where $\wh{Q}_K$ is the finite measure on $\wh{W}$ specified by
\begin{equation}\label{1.7}
\begin{array}{rl}
{\rm i)} & \wh{Q}_K(Z_0 = x) = e_K(x), \;\mbox{for $x \in E$},
\\[1ex]
{\rm ii)} & \mbox{when $e_K(x) > 0$, conditionally on $Z_0 = x$, $(Z_n)_{n \ge 0}$, $(Z_{-n})_{n \ge 0}$, $(\sigma_n)_{n \in \IZ}$}\\
&\mbox{are independent, respectively distributed as simple random walk starting}
\\
&\mbox{at $x$, as simple random walk starting at $x$ conditioned never to return}
\\
&\mbox{to $K$, and as a doubly infinite sequence of i.i.d. exponential variables with}
\\
&\mbox{parameter $1$}.
\end{array}
\end{equation}  

\n
As in \cite{Szni11c}, the canonical continuous-time random interlacement point process is then constructed similarly to (1.16) of \cite{Szni10a}, or (2.10) of \cite{Teix09b}, on a space $(\Omega, \cA, \IP)$, with $\o = \sum_{i \ge 0} \delta_{(\wh{w}_i^*,u_i)}$ denoting a generic element of $\Omega$. A central object of interest in this note is the random field of occupation times of random interlacements at level $u \ge 0$:
\begin{equation}\label{1.8}
\begin{split}
L_{x,u}(\o) = & \; \mbox{\f $\dis\frac{1}{\lambda_x}$} \;\dsl_{i \ge 0} \; \dsl_{n \in \IZ} \sigma_n(\wh{w}_i) \,1\{Z_n(\wh{w}_i) = x,u_i \le u\}, \;\mbox{for $x \in E$, $\o \in \Omega$}, 
\\
& \; \mbox{where $\o = \dsl_{i \ge 0} \delta_{(\wh{w}_i^*,u_i)}$ and $\pi^*(\wh{w}_i) = \wh{w}^*_i$, for each $i \ge 0$}.
\end{split}
\end{equation}

\n
The Laplace transform of $(L_{x,u})_{x \in E}$ has been computed in \cite{Szni11c}. More precisely, given a function $f$: $E \r \IR$, such that $\sum_{y \in E} g(x,y) | f(y) | < \infty$, for $x \in E$, one sets
\begin{equation}\label{1.9}
Gf(x) = \dsl_{y \in E} g (x,y) \,f(y), \;\mbox{for $x \in E$}.
\end{equation}

\n
One knows from Theorem \ref{theo2.1} and Remark 2.4 4) of \cite{Szni11c}, that when $V$: $E \r \IR_+$ has finite support and
\begin{equation}\label{1.10}
\sup\limits_{x \in E} GV(x) < 1,
\end{equation}
one has the identity
\begin{equation}\label{1.11}
\IE\Big[\exp\Big\{- \dsl_{x \in E} V(x) \,L_{x,u}\Big\}\Big] = \exp\{ - u \langle V, (I+GV)^{-1} \,1_E \rangle\}, \;\mbox{for $u \ge 0$},
\end{equation}

\n
where the notation $\langle f,g\rangle$ stands for $\sum_{x \in E} f(x) \,g(x)$, when $f,g$ are functions on $E$ such that the previous sum converges absolutely, and $1_E$ denotes the constant function identically equal to $1$ on $E$.

\section{An approximation scheme for random interlacements}
\setcounter{equation}{0}

In this section we develop an approximation scheme for $(L_{x,u})_{x \in E}$ in terms of the fields of local times of certain finite state space Markov chains. The main result is Theorem \ref{theo2.1}, but Remark \ref{rem2.2} states a by-product of the approximation scheme concerning the random interlacement at level $u$. This has a similar flavor to Theorem 4.17 of \cite{Szni11d}, where one gives one of several possible meanings to random interlacements viewed as ``Markovian loops going through infinity'', see also Le~Jan \cite{Leja11}, p.~85.

\medskip
We consider a non-decreasing sequence $U_n, n \ge 1$, of finite connected subsets of $E$, increasing to $E$, as well as $x_*$ some fixed point not belonging to $E$. We introduce the sets $E_n = U_n \cup \{x_*\}$, for $n \ge 1$, and endow $E_n$ with the weights $c^n_{x,y}$, $x,y \in E_n$, obtained by ``collapsing $U^c_n$ on $x_*$'', that is, for any $n \ge 1$, and $x,y \in U_n$, we set
\begin{equation}\label{2.1}
\begin{split}
c^n_{x,y} & = c_{x,y},
\\
c^n_{x_*,y} & = c^n_{y,x_*} = \dsl_{z \in E \backslash U_n} \,c_{z,y},
\end{split}
\end{equation}
and otherwise set $c^n_{x,y} = 0$ (i.e. $c^n_{x_*,x_*} = 0$). We also write
\begin{equation}\label{2.2}
\lambda^n_x = \dsl_{y \in E_n} c^n_{x,y}, \;\mbox{for $x \in E_n$  (in particular $\lambda^n_x = \lambda_x$, when $x \in U_n$)}.
\end{equation}

\n
We tacitly view $U_n$ as a subset of both $E$ and $E_n$. We consider the canonical simple random walk in continuous time on $E_n$, attached to the weights $c^n_{x,y}$, $x,y \in E_n$, with jump rate equal to $1$. We write $X^n_t$, $t \ge 0$, for its canonical process, $P^n_x$ for its canonical law starting from $x \in E_n$, and $E^n_x$ for the corresponding expectation.

\medskip
The local time of this Markov chain is defined by
\begin{equation}\label{2.3}
\ell^{n,x}_t = \mbox{\f $\dis\frac{1}{\lambda^n_x}$} \;\dis\int^t_0 1 \{X^n_s = x\} \,ds, \;\; \mbox{for $x \in E_n$ and $t \ge 0$}.
\end{equation}

\n
The function $t \ge 0 \r \ell_t^{n,x} \ge 0$ is continuous, non-decreasings, starts at $0$, and $P^n_y$-a.s. tends to infinity, as $t$ goes to infinity (the walk on $E_n$ is irreducible and recurrent). By convention, when $x \in E \backslash U_n$, we set $\ell_t^{n,x} = 0$, for all $t \ge 0$. We introduce the right-continuous inverse of $\ell_\point^{n,x_*}$
\begin{equation}\label{2.4}
\tau^n_u = \inf\{ t \ge 0; \ell_t^{n,x_*} > u\}, \; \mbox{for any $u \ge 0$}.
\end{equation}

\n
We are now ready for the main result of this section. We tacitly endow $\IR^E$ with the product topology, and convergence in distribution, as stated below (and in the sequel), corresponds to convergence in law of all finite dimensional marginals.

\begin{theorem}\label{theo2.1} $(u \ge 0)$
\begin{equation}\label{2.5}
\mbox{$(\ell^{n,x}_{\tau^n_u})_{x \in E}$ under $P_{x_*}^n$ converges in distribution to $(L_{x,u})_{x \in E}$ under $\IP$.}
\end{equation}
\end{theorem}

\begin{proof}
We give two proofs.

\medskip\n
{\it First proof:} We denote by $\cT$ the set of piecewise-constant, right-continuous, $E \cup \{x_*\}$-valued trajectories, which at a finite time reach $x_*$, and from that time onwards remain equal to $x_*$. We endow $\cT$ with its canonical $\sigma$-algebra.

\medskip
Under $P^n_{x_*}$, one has almost surely two infinite sequences $R_\ell, \ell \ge 1$ and $D_\ell, \ell \ge 1$,
\begin{equation}\label{2.6}
R_1 = 0 < D_1 <  R_2 < \dots < R_\ell < D_\ell < \dots
\end{equation}
of successive returns $R_\ell$ of $X_\point^n$ to $x_*$, and departures $D_\ell$ from $x_*$,  which tend to infinity. One introduces the random point measure on $\cT$
\begin{equation}\label{2.7}
\Gamma^n_u = \dsl_{\ell \ge 1} 1\{D_\ell < \tau^n_u\} \,\delta_{(X^n_{D_\ell + \cdot})_{0 \le \cdot \le R_{\ell + 1} - D_\ell}}, \; u \ge 0,
\end{equation}

\n
which collects the successive excursions of $X_\point^n$ (out of $x_*$ until first return to $x_*$) that start before $\tau^n_u$. By classical Markov chain excursion theory we know that
\begin{equation}\label{2.8}
\begin{array}{l}
\mbox{$\Gamma^n_u$ is a Poisson point measure on $\cT$ with intensity measure}
\\
\gamma^n_u(\cdot) = u \,P_{\kappa_n}^n [(X_{s \wedge T_{U_n}}^n)_{s \ge 0} \in \cdot] \;\mbox{on $\cT$},
\end{array}
\end{equation}
where $T_{U_n}$ stands for the exit time of $X_\point^n$ from $U_n$ and $\kappa_n$ for the measure on $U_n$
\begin{equation}\label{2.9}
\kappa_n(y) = \lambda^n_{x_*} \;\dis\frac{c^n_{x_*,y}}{\lambda^n_{x_*}} = c^n_{x_*,y} \stackrel{(\ref{2.1})}{=} \dsl_{x \in E \backslash U_n} c_{x,y}, \;\mbox{for $y \in U_n$}.
\end{equation}

\n
When starting in $U_n$,the Markov chains $X$ on $E$, and $X^n$ on $E_n$, have the same evolution strictly before the exit time of $U_n$. Denoting by $(X_\point)_{0 \le \cdot < T_{U_n}}$ the random element of $\cT$, which equals $X_s$, for $0 \le s < T_{U_n}$, and $x_*$ for $s \ge T_{U_n}$, we see that
\begin{equation}\label{2.10}
\gamma^n_u(\cdot) = u \,P_{\kappa_n} [(X_\point)_{0 \le \cdot < T_{U_n}}  \in \cdot ], \; \mbox{for all $n \ge 1$, $u \ge 0$}.
\end{equation}

\n
Let $K$ be a finite subset of $E$, and assume $n$ large enough so that $K \subseteq U_n$. We introduce the point measure on $\cT$ obtained by selecting the excursions in the support of $\Gamma^n_u$ that enter $K$, and only keeping track of their trajectory after they enter $K$, that is
\begin{equation}\label{2.11}
\mu^n_{K,u} = \theta_{H_K} \circ (1 \{H_K < \infty\} \, \Gamma_u^n),
\end{equation}

\medskip\n
where $\theta_t$, $t \ge 0$, stands for the canonical shift on $\cT$, and we use similar notation on $\cT$ as below (\ref{1.1}). By (\ref{2.8}), (\ref{2.10}) it follows that 
\begin{equation}\label{2.12}
\begin{array}{l}
\mbox{$\mu^n_{K,u}$ is a Poisson point measure on $\cT$ with intensity measure}\\
\gamma^n_{K,u}(\cdot) = u \,P_{\rho^n_K} [(X_\point)_{0 \le \cdot < T_{U_n}} \in \cdot]  \; \mbox{on} \; \cT,
\end{array}
\end{equation}

\smallskip\n
where $\rho_K^n$ is the measure supported by $K$ such that 
\begin{equation}\label{2.13}
\rho^n_K(x) = P_{\kappa_n} [H_K < T_{U_n}, X_{H_K} = x] = e_{K,U_n}(x), \;\mbox{for $x \in K$},
\end{equation}

\n
where the last equality follows from (1.60) in Proposition 1.8 of \cite{Szni11d}. Note that $e_{K,U_n}$ and $e_K$ are concentrated on $K$, and for $x \in K$, 
\begin{equation}\label{2.14}
e_{K,U_n}(x) \stackrel{(\ref{1.4})}{=} P_x [\wt{H}_K > T_{U_n}] \,\lambda_x \underset{n \r \infty}{\longrightarrow} P_x [\wt{H}_K = \infty] \,\lambda_x = e_K(x).
\end{equation}

\n
Consider $V$: $E \r \IR_+$ supported in $K$, and $\Phi$: $\cT \r \IR_+$, the map 
\begin{equation*}
\Phi(w) = \dsl_{x \in E} V(x) \;\mbox{\f $\dis\frac{1}{\lambda_x}$} \;\dis\int^\infty_0 1\{w(s) = x\} ds, \; 
\mbox{for $w \in \cT$}.
\end{equation*}

\n
The measure $\mu^n_{K,u}$ contains in its support the pieces of the trajectory $X^n_\point$ up to time $\tau^n_u$, where $X^n_\point$ visits $K$, see (\ref{2.11}), and we have
\begin{equation}\label{2.15}
\begin{array}{l}
E^n_{x_*} \Big[\exp\Big\{ - \dsl_{x \in E} V(x) \,\ell^{n,x}_{\tau^n_u}\Big\}\Big] = E^n_{x_*} \Big[\exp\Big\{ - \langle \mu^n_{K,u}, \Phi \rangle \Big\}\Big] \stackrel{(\ref{2.12})}{=}
\\
\exp\Big\{\dis\int_\cT (e^{-\Phi} -1) \,d \gamma^n_{K,u}\Big\} \stackrel{(\ref{2.12}), (\ref{2.13})}{=} \exp\Big\{u \,E_{e_{K,U_n}} \Big[e^{-\int_0^{T_{U_n}} \frac{V}{\lambda} (X_s) ds}-1\Big]\Big\}
\\
\underset{n \r \infty}{\longrightarrow} \exp\Big\{ u\,E_{e_K} \big[ e^{-\int^\infty_0 \frac{V}{\lambda} (X_s) ds} -1\big]\Big\} = \IE\Big[\exp\Big\{- \dsl_{x \in E} V(x)\,L_{x,u}\Big\}\Big],
\end{array}
\end{equation}

\n
where we used (\ref{2.14}) and the fact that $T_{U_n} \uparrow \infty$, $P_x$-a.s., for $x$ in $E$, for the limit in the last line, and a similar calculation as in (2.5) of \cite{Szni11c} for the last equality. Since $K$ and the function $V$: $E \r \IR_+$, supported in $K$, are arbitrary, the claim (\ref{2.5}) follows.

\bigskip\n
{\it Second Proof:} We will now make direct use of (\ref{1.11}). The argument is more computational, but also of interest. We consider $K$ and $V$ as above, as well as a positive number $\lambda$. We assume $n$ large enough so that $K \subseteq U_n$. We further make a smallness assumption on the non-negative function $V$ (supported in $K$):
\begin{equation}\label{2.16}
\sup\limits_{x \in E}\, (GV) (x) + \lambda^{-1} \dsl_{x \in K} V(x) < 1.
\end{equation}

\n
We define the operator $G_n$ on $\IR^{E_n}$ attached to the kernel $g_n(\cdot,\cdot)$ in a similar fashion to (\ref{1.9}), where we use the notation
\begin{equation}\label{2.17}
g_n(x,y) = g_{U_n} (x,y) + \lambda^{-1}, \;\mbox{for $x,y \in E_n$},
\end{equation}

\n
and we have set $g_{U_n}(x_*,\cdot) = g_{U_n}(\cdot,x_*) = 0$, by convention, to define $g_{U_n}(\cdot,\cdot)$ on $E_n \times E_n$.

\medskip
Since $g_{U_n}(\cdot,\cdot) \le g(\cdot,\cdot)$ on $E \times E$, it follows from (\ref{2.16}) that $\sup_{x \in E_n} (G_nV)(x) < 1$, where we have set $V(x_*) = 0$, by convention, so that the operator $I + G_n V$ is invertible.

\medskip
We introduce the positive number
\begin{equation}\label{2.18}
a_n = \dis\int^\infty_0 \lambda e^{-\lambda u} E^n_{x_*} \Big[e^{-\sum\limits_{x \in E} V(x) \ell^{n,x}_{\tau^n_u}}\Big]\,du,
\end{equation}

\n
where we recall that $\ell^{n,x}_t = 0 $, when $x \in E \backslash U_n$.  Using (2.93), (2.41), (2.71) of \cite{Szni11d}, or by (8.44) and Remark 3.10.3 of Marcus-Rosen \cite{MarcRose06}, we know that 
\begin{equation}\label{2.19}
a_n = ( I + G_n V)^{-1} 1_{E_n}(x_*).
\end{equation}

\n
We then define the function $h_n$ on $E_n$ and the real number $b_n$:
\begin{equation}\label{2.20}
h_n = (I+G_n  V)^{-1} \,1_{E_n} \;\mbox{and} \;b_n = \dsl_{x \in K} V(x) \,h_n(x).
\end{equation}

\n
We let $G^*_{U_n}$ be the operator on $\IR^{E_n}$ attached to the kernel $g_{U_n}(\cdot,\cdot)$ (on $E_n \times E_n$), in a similar fashion to (\ref{1.9}). By (\ref{2.17}) and (\ref{2.20}), we have
\begin{equation}\label{2.21}
h_n + G^*_{U_n} V h_n + \lambda^{-1} b_n \,1_{E_n} = 1_{E_n}, \;\mbox{so that}
\end{equation}
\begin{equation*}
h_n = \Big(1 - \mbox{\f $\dis\frac{b_n}{\lambda}$}\Big) (1 + G^*_{U_n} V)^{-1} 1_{E_n},
\end{equation*} 

\smallskip\n
noting that the above inverse is well defined by the same argument used below (\ref{2.17}). By the second equality in (\ref{2.20}) it follows  that
\begin{equation}\label{2.22}
b_n = \Big(1 - \mbox{\f $\dis\frac{b_n}{\lambda}$}\Big) \, \dsl_{x \in K} V(x) (I + G^*_{U_n} V)^{-1}(x) = \big(1 - \mbox{\f $\dis\frac{b_n}{\lambda}$}\Big) \langle V, (I+G_{U_n} V)^{-1} 1_E \rangle,
\end{equation}

\n
where we refer to below (\ref{1.11}) for notation, $G_{U_n}$ is the operator on $\IR^E$ attached to the kernel $g_{U_n}(\cdot,\cdot)$ on $E \times E$, and the last equality follows by writing the Neumann series for $(I+G^*_{U_n} V)^{-1}$ and $(I+G_{U_n} V)^{-1}$.

\medskip
We can now solve for $b_n$. Noting that $a_n = h_n(x_*) = 1 - \frac{b_n}{\lambda}$, by (\ref{2.21}), we find
\begin{equation}\label{2.23}
a_n = (1 + \lambda^{-1} \langle V, (I+G_{U_n} V)^{-1} 1_E \rangle)^{-1}.
\end{equation}

\n
Using the Neumann series for $(I + G_{U_n}V)^{-1}$, and applying dominated convergence together with the fact that $g_{U_n} (\cdot,\cdot) \uparrow g(\cdot,\cdot)$ on $E \times E$, we see that
\begin{equation}\label{2.24}
a_n \underset{n \r \infty}{\longrightarrow} (1 + \lambda^{-1} \langle V, (I+GV)^{-1} 1_E \rangle)^{-1}.
\end{equation}

\n
Taking the identity (\ref{1.11}) into account, we have shown that under (\ref{2.16}),
\begin{equation}\label{2.25}
\lim\limits_n \dis\int^\infty_0 \lambda e^{-\lambda u} E^n_{x_*} \big[e^{-\sum\limits_{x \in E} V(x) \ell^{n,x}_{\tau^n_u}}\big] \,du = \dis\int^\infty_0 \lambda e^{-\lambda u} \IE\big[e^{-\sum\limits_{x \in E} V(x) L_{x,u}}\big]\,du.
\end{equation}

\medskip\n
Note that when $V$: $E \r \IR_+$ is supported in $K$ and $\sup_{x \in E} GV(x) < 1$, then (\ref{2.16}) holds for $\lambda$ large (depending on $V$). The expectation under the integral in the left-hand side of (\ref{2.25}) is non-increasing in $u$, whereas the expectation under the integral in the right-hand side of (\ref{2.25}) is continuous in $u$ by (\ref{1.11}). It then follows from \cite{Chun74}, p.~193-194, that for $V$ as above,
\begin{equation}\label{2.26}
\lim\limits_n E^n_{x_*} \big[e^{-\sum\limits_{x \in E} V(x) \ell^{n,x}_{\tau^n_u}}\big] = \IE \big[ e^{-\sum\limits_{x \in E} V(x) L_{x,u}}\big], \;\mbox{for $u \ge 0$}.
\end{equation}

\n
This readily implies the tightness of the laws of $(\ell^{n,x}_{\tau^n_u})_{x\in K}$ under $P^n_{x_*}$, and uniquely determines the Laplace transform of their possible limit points, see Theorem 6.6.5 of \cite{Chun74}. Letting $K$ vary, the claim (\ref{2.5}) follows.
\end{proof}

\begin{remark}\label{rem2.2}  \rm
The approximation scheme introduced in this section can also be used to approximate the random interlacement at level $u$, as we now explain. We let $\cI^n_u$ stand for the trace left on $U_n$ by the walk on $E_n$ up to time $\tau^n_u$:
\begin{equation}\label{2.27}
\cI^n_u = \{x \in U_n; \, \ell^{n,x}_{\tau^n_u} > 0\}.
\end{equation}

\n
By (\ref{2.12}), (\ref{2.14}), it follows that for any finite subset $K$ of $E$ and $u \ge 0$,
\begin{equation}\label{2.287}
P^n_{x_*} [\cI^n_u \cap K = \phi] = P^n_{x_*} [\mu^n_{K,u} = 0] = e^{-u\, {\rm cap}_{U_n}(K)} \underset{n}{\stackrel{(\ref{1.4}),(\ref{1.5})}{\longrightarrow}} e^{-u\, {\rm cap}(K)} = \IP [\cI^u \cap K = \phi],
\end{equation}

\n
where $\cI^u$ stands for the random interlacement at level $u$, that is, the trace on $E$ of doubly infinite trajectories modulo time-shift in the Poisson cloud $\omega$ with label at most $u$. By an inclusion-exclusion argument, see for instance Remark 4.15 of \cite{Szni11d} or Remark 2.2 of \cite{Szni10a}, it follows that, as $n \r \infty$,
\begin{equation}\label{2.29}
\mbox{$\cI^n_u$ under $P^n_{x_*}$, converges in distribution to $\cI^u$ under $\IP$, for any $u \ge 0$,}
\end{equation}

\medskip\n
where the above distributions are viewed as laws on  $\{0,1\}^E$ endowed with the product topology. \hfill $\square$

\end{remark}

\section{Proof of the isomorphism theorem}
\setcounter{equation}{0}

In this short section we combine Theorem \ref{theo2.1} and the generalized second Ray-Knight theorem of \cite{EisKasMarRosShi00} to prove Theorem \ref{theo0.1}. We also state a variation of (\ref{0.4}) in Remark \ref{rem3.1}.

\bigskip\n
{\it Proof of Theorem \ref{theo0.1}:} For $U \subseteq G$ we denote by $P^{G,U}$ the law on $\IR^E$ of the centered Gaussian field with covariance $E^{G,U} [\varphi_x \varphi_y] = g_U(x,y)$, $x,y \in E$ (in particular $\varphi_x = 0$, $P^{G,U}$-a.s., when $x \in E \backslash U$). It follows from the generalized second Ray-Knight theorem, see Theorem 8.2.2 of \cite{MarcRose06}, or Theorem 2.17 of \cite{Szni11d}, that for $n \ge 1$, $u \ge 0$, in the notation of Section 2,
\begin{equation}\label{3.1}
\begin{array}{l}
\big(\ell^{n,x}_{\tau^n_{u}} + \fr \;\varphi_x^2\big)_{x \in U_n} \;\; \mbox{under $P^n_{x_*} \otimes P^{G,U_n}$, has the same law as}
\\[1ex]
\big(\fr \; (\varphi_x + \sqrt{2u})^2\big)_{x \in U_n} \;\; \mbox{under $P^{G,U_n}$}.
\end{array}
\end{equation}

\n
Since $g_{U_n}(\cdot,\cdot) \uparrow g_U(\cdot,\cdot)$, we see that $P^{G,U_n}$ converges weakly to $P^G$ (looking for instance at characteristic functions of finite dimensional marginals). Taking Theorem \ref{theo2.1} into account we thus see letting $n$ tend to infinity that
\begin{equation}\label{3.2}
\begin{array}{l}
\big(L_{x,u} + \fr \;\varphi_x^2\big)_{x \in E} \;\; \mbox{under $P \otimes P^G$, has the same law as}
\\[1ex]
\big(\fr \; (\varphi_x + \sqrt{2u})^2\big)_{x \in E} \;\; \mbox{under $P^G$},
\end{array}
\end{equation}
and Theorem \ref{theo0.1} is proved. \hfill $\square$

\begin{remark}\label{rem3.1} \rm Let us mention a variation on (\ref{0.4}) of Theorem \ref{theo0.1}. By Theorem 1.1 of \cite{EisKasMarRosShi00}, one knows that for $u \ge 0$, $a \in \IR$, $n \ge 1$,
\begin{equation}\label{3.3}
\begin{array}{l}
\big(\ell^{n,x}_{\tau^n_u} + \fr \;(\varphi_x + a)^2\big)_{x \in U_n} \;\;\mbox{under $P^n_{x_*} \otimes P^{G,U_n}$, has the same law as}
\\[1ex]
\big(\fr \;\big(\varphi_x + \sqrt{2u + a^2}\big)^2\big)_{x \in U_n} \;\;\mbox{under} \;P^{G,U_n} .
\end{array}
\end{equation}

\n
Letting $n$ tend to infinity, the same argument as above shows that for $u \ge 0$, and $a \in \IR$,
\begin{equation}\label{3.4}
\hspace{-3ex} \begin{array}{l}
\big(L_{x,u} + \fr \;(\varphi_x + a)^2\big)_{x \in E} \;\;\mbox{under $\IP \otimes P^G$, has the same law as}
\\[1ex]
\big(\fr \;\big(\varphi_x + \sqrt{2u + a^2}\big)^2\big)_{x \in E} \;\; \mbox{under} \;P^G .
\end{array}
\end{equation}
\hfill $\square$
\end{remark}

\section{An application}
\setcounter{equation}{0}

We illustrate the use of Theorem \ref{theo0.1} and show how one can study the large $u$ asymptotics of $(L_{x,u})_{x \in E}$ and in particular recover Theorem 5.1 of \cite{Szni11c}, see also Remark 5.2 of \cite{Szni11c}. We denote by $x_0$ some fixed point of $E$. 

\begin{theorem}\label{theo4.1} As $u \r \infty$,
\begin{align}
&\mbox{$\Big(\mbox{\f $\dis\frac{1}{u}$} \;L_{x,u}\Big)_{x \in E}$ converges in distribution to the constant field equal to $1$}, \label{4.1}
\\[1ex]
&\Big(\mbox{\f $\dis\frac{L_{x,u} - u}{\sqrt{2u}}$}\Big)_{x \in E} \;\mbox{converges in distribution to $(\varphi_x)_{x \in E}$ under $P^G$.} \label{4.2a}
\intertext{In particular, as $u \r \infty$,}
&\Big(\mbox{\f $\dis\frac{L_{x,u} - L_{x_0,u}}{\sqrt{2u}}$}\Big)_{x \in E} 
\;\mbox{converges in distribution to $(\varphi_x - \varphi_{x_0})_{x \in E}$ under $P^G$.} \label{4.2}
\end{align}
\end{theorem}

\begin{proof}
We first prove (\ref{4.1}). To this end we note that $P^G$-a.s., for $x \in E$,
\begin{equation}\label{4.3}
\mbox{$\mbox{\f $\dis\frac{1}{2u}$} \;\varphi^2_x \r 0$ and $\mbox{\f $\dis\frac{1}{2u}$} \;(\varphi_x + \sqrt{2u})^2 \r 1$, as $u \r \infty$.}
\end{equation}

\n
Thus Theorem \ref{theo0.1} implies that $\frac{1}{u} \; L_{x,u}$ converges in distribution to the constant $1$ as $u$ tends to infinity, and (\ref{4.1}) follows.

\medskip
We then observe that (\ref{4.2}) is a direct consequence of (\ref{4.2a}), and turn to the proof of (\ref{4.2}). Note that by Theorem \ref{theo0.1}
\begin{equation}\label{4.4}
\begin{array}{l}
\Big(\mbox{\f $\dis\frac{L_{x,u} - u}{\sqrt{2u}}$} + \mbox{\f $\dis\frac{1}{2 \sqrt{2u}}$} \;\varphi^2_x\Big)_{x \in E} \;\mbox{under $\IP \otimes P^G$, has the same law as}
\\[2ex]
\Big(\mbox{\f $\dis\frac{1}{2 \sqrt{2u}}$}\;[(\varphi_x + \sqrt{2u})^2 -2u]\Big)_{x \in E}.
\end{array}
\end{equation}

\n
Note also that for each $x \in E$, $P^G$-a.s., as $u \r \infty$,
\begin{align}
&\mbox{\f $\dis\frac{1}{2 \sqrt{2u}}$}\; \varphi^2_x  \r 0, \;\mbox{and} \label{4.5}
\\[1ex]
&\mbox{\f $\dis\frac{1}{2 \sqrt{2u}}$}\; [(\varphi_x + \sqrt{2u})^2 -2u] = \mbox{\f $\dis\frac{1}{2 \sqrt{2u}}$}\;  \varphi^2_x  + \varphi_x  \r \varphi_x . \label{4.6}
\end{align}

\n
Looking at the characteristic function of finite dimensional marginals of the fields in the first and second line of (\ref{4.4}), we readily obtain (\ref{4.2}).
\end{proof}

\begin{remark}\label{rem4.2} \rm 
In view of the above illustration of the use of Theorem \ref{theo0.1}, one can naturally wonder about the nature of its scope as a transfer mechanism between random interlacements and the Gaussian free field. \hfill $\square$
\end{remark}

\end{document}